\documentclass[12pt,twoside, final]{amsart}
\UseRawInputEncoding
\usepackage[utf8]{inputenc}
\pagespan{1}{\pageref*{LastPage}}
\usepackage{etoolbox,lastpage}
\commby{}
\usepackage{amsmath,amsfonts,amssymb,enumerate}
\usepackage{graphicx}		
\usepackage{color}
\usepackage[colorlinks]{hyperref}
 
\newtheorem{theorem}{Theorem}[section]
\newtheorem{proposition}[theorem]{Proposition}
\newtheorem{lemma}[theorem]{Lemma}
\newtheorem{corollary}[theorem]{Corollary}
\theoremstyle{definition}
\newtheorem{definition}[theorem]{Definition}
\newtheorem{example}[theorem]{Example}
\theoremstyle{remark}
\newtheorem{remark}[theorem]{Remark}
\numberwithin{equation}{section}

\usepackage[all]{xy}
 
\begin{document}

 
\title[ The homogeneous spectrum of a $\Bbb Z_2$-graded comm. ring ]{The homogeneous spectrum of a $\Bbb Z_2$-graded commutative ring } 
 
\author[M. Aqalmoun]{Mohamed Aqalmoun}
\address[Mohamed Aqalmoun]{Department of Mathematics, Higher Normal School, Sidi Mohamed Ben Abdellah University, Fez, Morocco.}
\email{ maqalmoun@yahoo.fr}


%
 
 \maketitle
%

\begin{abstract}
Let $\Bbb Z_2:=\Bbb Z/2\Bbb Z$ be the additive group with two elements. In this article, we focus only on  $\Bbb Z_2$-graded commutative ring i.e commutative ring $R$ such that $R=R_0\oplus R_1$ as Abelian group and  $R_iR_j\subseteq R_{i+j}$ for all $i,j\in \Bbb Z_2$. Our main goals is to establish a strong relation between  $\Bbb Z_2$-graded prime ( maximal ) ideals of $R$ and prime ( maximal) ideals of $R_0$, for instance, it is showed that, the $\Bbb Z_2$-graded spectrum of $R$ is homeomorphic to the spectrum of $R_0$ with respect  to the Zariski topologies.     
\end{abstract}
 
\section{\bf Introduction}
Let $G$ be an Abelian group with identity $e$ and $R$ be a commutative ring with unit. Then $R$ is called a $G$-graded ring if  there exist additive subgroups $R_g$ of $R$ indexed by elements $g\in G$ such that $R=\bigoplus_{g\in G}R_g$ and $R_gR_{g'}\subseteq R_{gg'}$ for all $g,g'\in G$, where $R_gR_{g'}$ consists of the finite sums of ring products $ab$ with $a\in R_g$ and $b\in R_{g'}$. The elements of $R_g$ are called homogeneous elements of $R$ of degree $g$. The homogeneous elements of the ring $R$ are denoted by $h(R)$, i.e. $h(R)=\cup_{g\in G}R_g$. If $a\in R$, then the element $a$ can be written uniquely as $\sum_{g\in G}a_g$, where $a_g\in R_g$ is called the $g$-component of $a$ in $R_g$.\par 
Let $R$ be a $G$-graded ring and $I$ be an ideal of $R$. Then $I$ is called a graded ideal of $R$ if $I=\bigoplus_{g\in G}(I\cap R_g)$. If $I$ is a $G$-graded ideal of $R$, then the quotient ring $R/I$ is a $G$-graded ring. Indeed $R/I=\bigoplus_{g\in G}(R/I)_g$ where $(R/I)_g=(R_g+I)/I=\{x+I\ / x\in R_g\}$. Let $S\subseteq h(R)$ be a multiplicatively closed subset of $R$. Then the ring of fractions $S^{-1}R$ is a $G$-graded ring. Indeed $S^{-1}R=\bigoplus_{g\in G}(S^{-1}R)_g$ where $(S^{-1}R)_g=\{r/s\ \ / \ r\in h(R), s\in S \text{ and } \deg r=g\deg s \}$.\par 
A $G$-graded ideal $P$ of a $G$-graded ring $R$ is called $G$-graded prime ideal of $R$ if $P\ne R$ and if whenever $r$ and $s$ are homogeneous elements of $R $ such that $rs\in P$, then either $r\in P$ or $s\in P$. The $G$-graded spectrum or homogeneous spectrum of $R$ is the set of all $G$-graded prime ideals of $R$, it is denoted by $G\mathrm{Spec}R$ . A graded ideal $M$ of $R$ is said to be graded maximal ideal of $R$ if $M\ne R$ and if $J$ is a graded ideal of $R$ such that $M\subseteq J\subseteq R$, then $J=M$ or $J=R$, the set of all $G$-graded maximal ideals of $R$ is denoted by $G\mathrm{Max}R$.  \par 
For a commutative ring $R$ and ideal $I$  of $R$, the variety of $I$ is the subset $V(I):=\{P\in \mathrm{Spec}R\ / I\subseteq P\}$. Then the collection $\{V(I)\ / I \text{ ideal of } R\}$ satisfies the axioms for the closed sets of a topology on $\mathrm{Spec}R$, called the Zariski topology. Note that for each ideal $I$, the set $\mathrm{Spec}R-V(I)$ is an open subset, it is denoted $D(I)$. Recall that, the collection $D(f)$ where $f$ runs through $R$ forms a basis of opens. Similarly, let $R$ be a $G$-graded commutative ring and $I$ a $G$-graded ideal of $R$, the $G$-variety of $I$ is the subset $V_G(I):=\{P\in G\mathrm{Spec}R\ / I\subseteq P\}$. The collection $\{V_G(I)\ / I \text{ is a } G\text{-graded ideal of } R\}$ satisfies the axioms for the closed sets of a topology on $G\mathrm{Spec}R$, called the Zariski topology, see  \cite{Refai}, \cite{Dara}, \cite{Daran}. Note that,   the collection $D(r)$ where $r$ runs through $h(R)$ forms a basis of opens for the Zariski topology.
\par 
Let $G=\Bbb Z_2=\Bbb Z/2\Bbb Z$ and $R$ be $\Bbb Z_2$-graded commutative ring. Then $R=R_0\oplus R_1$, where $R_0$ is a subring of $R$ and $R_1$ is an $R_0$-module such that $R_1R_1\subseteq R_0$.  Note that $R_1R_1$  is an ideal of $R_0$, we denote it by $R_1^2$. Similarly $R_1^3$ is an $R_0$-submodule of $R_1$. If $P$ is a $\Bbb Z_2$-graded prime ideal of $R$, then $P\cap R_0$ is a prime ideal of $R_0$. This yields a natural map $\Bbb Z_2\mathrm{Spec}R\to \mathrm{Spec}R_0$. 
\begin{example}\label{NGradedPolynome}
Let $A$ be a commutative ring, then the polynomial ring $R=A[X]$ is $\Bbb Z_2$-graded ring with the grading  $R_0= \oplus_nA(X^{2n})$ and $R_1= \oplus_nA(X^{2n+1})$.  In fact every $\Bbb N$-graded (respectively $\Bbb Z$-graded ) commutative ring $R$ has a natural $\Bbb Z_2$-grading with $R_0'=\oplus_nR_{2n}$ and $R'_1=\oplus_nR_{2n+1}$.
\end{example} 
\begin{example}
Let $A$ be a commutative ring and $X=(X_{\alpha})_{\alpha}$ be a set of indeterminate over $A$ and $(a_{\alpha \beta})_{\alpha,\beta}$ be a family of elements of $A$. Consider the ideal $I$ generated by all elements $X_{\alpha}X_{\beta}-a_{\alpha\beta}$. Then $R:=\dfrac{A[X]}{I}$ is a $\Bbb Z_2$-graded ring with the grading $R_0=A$ and $R_1 $ is  the $A$-submodule of $R$ generated by all $X_{\alpha}+I$.
\end{example} 
\begin{remark} \textit{Matrix ring representation of a $\Bbb Z_2$-graded commutative ring.} Let $R$ be a $\Bbb Z_2$-graded commutative ring. If $r=r_0+r_1, s=s_0+s_1\in R=R_0\oplus R_1$, then $rs=(r_0s_0+r_1s_1)+(r_0s_1+s_0r_1)$, that  is $(rs)_0=r_0s_0+r_1s_1$ and $(rs)_1=r_0s_1+s_0r_1$. So, one can identifies $R$ with the matrix ring $$R'=\left\{ \begin{pmatrix}
r_0&r_1\\ r_1&r_0
\end{pmatrix}\ / \ r_0\in R_0 ,\ r_1\in R_1 \right\}$$ 
\end{remark}
\par The purpose of this paper is to establish a strong relation between $\Bbb Z_2$-graded prime ideals of a $\Bbb Z_2$-graded commutative ring $R$ and prime ideals of $R_0$ via the natural map $\Bbb Z_2\mathrm{Spec}R\to \mathrm{Spec}R_0 $. It is proved  that, this map is one to one, and in fact it is a homeomorphism with respect to the Zariski topologies. In fact, to fulfill this goal, we give the form of a $\Bbb Z_2$-graded prime ideal $P$ from its degree zero homogeneous part $P\cap R_0$, their expression lies on the fact that $P\cap R_0\in V(R_1^2)$ or $P\cap R_0\in D(R_1^2)$. Precisely, every $\Bbb Z_2$-graded ideal can be discovered from its degree zero homogeneous part.  
\section{\bf $\Bbb Z_2$-Graded ideals}
We start this section with some useful properties for $\Bbb Z_2$-graded ideals of a $\Bbb Z_2$-graded commutative ring. We will see that, when $R$ is a strongly $\Bbb Z_2$-graded commutative ring, those ideals are no thing but as ideals of $R_0$.    
\begin{proposition}\label{Gideal}
Let $R$ be a $\Bbb Z_2$-graded ring and $J$ be an ideal of $R$. The following statements are equivalent.
\begin{enumerate}
\item $J$ is a $\Bbb Z_2$-graded ideal of $R$.
\item $J=I+R'$ where $I$ is an ideal of $R_0$ and $R'$ is an $R_0$-submodule of $R_1$ with $IR_1\subseteq R'$ and $R_1R'\subseteq I$. 
\end{enumerate}
\end{proposition}
\begin{proof}
Let $J$ be a $\Bbb Z_2$-graded ideal of $R$. Then $J=I+R'$ where $I:=J\cap R_0$ is an ideal of $R_0$ and $R':=J\cap R_1$ is an $R_0$-submodule of $R_1$. If $x\in R_1$ and $y\in R'$, then $xy\in J$ is an element of degree $0$, so $xy\in J\cap R_0=I$. It follows that $R_1R'\subseteq I$. If $i\in I$ and $y\in R_1$, then $iy\in J$ is an element of degree $1$, so $iy\in J\cap R_1=R'$. Therefore $IR_1\subseteq R'$.\par 
Let $J=I+R'$ where $I$ is an ideal of $R_0$ and $R'$ is an $R_0$-submodule of $R_1$ with $IR_1\subseteq R'$ and $R_1R'\subseteq R'$. Clearly $J$ is a subgroup of $R$. Let $x=i+r'\in I+R'$ and $ y=a+r\in R=R_0+R_1$. Then $xy=(ia+r'r)+(ir+ar')\in I+R'=J$.
\end{proof}
If $R'$ is an $R_0$-submodule of $R_1$. Then the residual $R'$ by $R_1$ is defined as follow $$(R':_{R_0}R_1)=\{a\in R_0\ / aR_1\subseteq R'\}$$ It is an ideal of $R_0$. So, the condition $IR_1\subseteq R'$ is equivalent to $I\subseteq (R':_{R_0}R_1)$.
\begin{remark}\label{Class} As a consequence, we have the two classes of $\Bbb Z_2$-graded ideals, which will appear for our discussion. 
\begin{enumerate}
\item If $I$ is an ideal of $R_0$, then $I+IR_1$ is a $\Bbb Z_2$-graded ideal of $R$. Indeed, $IR_1\subseteq IR_1$ and $R_1(IR_1)= IR_1R_1\subseteq IR_0=I$. 
\item If $R'$ is a submodule of $R_1$, then $(R':_{R_0}R_1)+R'$ is a $\Bbb Z_2$-graded ideal of $R$. Indeed; $ (R':_{R_0}R_1)R_1\subseteq R'$ and $ (R'R_1)R_1\subseteq R'R_0\subseteq R'$ that is $R'R_1\subseteq (R':_{R_0}R_1)$.
\end{enumerate} 
\end{remark}
\begin{definition}
A $G$-graded commutative ring is called strongly graded if $R_gR_{g'}=R_{gg'}$ for all $g,g'\in G$. 
\end{definition}
Note that, if $R$ is a $\Bbb Z_2$-graded commutative ring, then $R$ is strongly graded if and only if $R_1^2=R_0$. \par 
The following result establish a relation between graded ideals of a strongly $\Bbb Z_2$-graded commutative ring $R$ and ideals of $R_0$. 
\begin{proposition}\label{SGideal}
Let $R$ be a strongly $\Bbb Z_2$-graded commutative ring. The following statements hold.
\begin{enumerate}
\item $R_1$ is a finitely generated $R_0$-module. 
\item $R_1$ is a multiplication module.
\item If $I$ and $I'$ are ideals of $R_0$ such that $IR_1\subseteq I'R_1$, then $I\subseteq I'$.
\item The $\Bbb Z_2$-graded ideals of $R$ are $ I+IR_1$ where $I$ is an ideal of $R_0$.
\end{enumerate}
\end{proposition}
\begin{proof}
\begin{enumerate}
\item Since $R_1^2=R_0$, there exist  $a_i,b_i\in R_1$ such that $1=\displaystyle\sum_{i=1}^na_ib_i$. Then $R_1=(a_1,\ldots,a_n)$. Indeed, if $x\in R_1$, then $x= \displaystyle\sum_{i=1}^n(xb_i)a_i$. The elements $xb_i$ are of degree zero, so they are in $R_0$. Therefore $R_1=(a_1,\ldots,a_n)$.
\item Let $R'$ be an $R_0$-submodule of $R_1$. Clearly, $(R':_{R_0}R_1)R_1\subseteq R'$. Let $x\in R'$, then $x=\displaystyle\sum_{i=1}^n(xb_i)a_i$. Each  $xb_i$ is an element of $(R':_{R_0}R_1)$, in fact if $c\in R_1$, then $(xb_i)c=(b_ic)x\in R'$ since $b_ic\in R_0$ and $x\in R'$. Thus $R'=(R':_{R_0}R_1)R_1$. Therefore $R_1$ is a multiplication module.
\item If $IR_1\subseteq I'R_1$, then $IR_1^2\subseteq I'R_1^2$, so $I\subseteq I'$.
\item Let $J$ be a $\Bbb Z_2$-graded ideal of $R$, then $J=I+R'$ where $I$ is an ideal of $R_0$ and $R'$ is an $R_0$-submodule of $R_1$ such that $IR_1\subseteq R' $ and $R'R_1\subseteq I$. Since $R_1$ is a multiplication module, there is an ideal  $I'$ of $R_0$ such that $R'=I'R_1$. But $IR_1\subseteq R'=I'R_1$, so $I\subseteq I'$. On other hand, $I'=I'R_1^2=R'R_1\subseteq I $. Thus $I=I'$ and $ J=I+IR_1$. Conversely, if $J=I+IR_1$ where $I$ is an ideal of $R_0$, then $J$ is  a $\Bbb Z_2$-graded ideal (see the Remark \ref{Class}). 
\end{enumerate}
\end{proof}
\begin{remark}
Let $R$ be a strongly $\Bbb Z_2$-graded ring. The map $J\mapsto J\cap R_0$ realize a bijection between the set of $\Bbb Z_2$-graded ideals of $R$ and the set of ideals of $R_0$.  
\end{remark}
One can show that, for a  strongly $\Bbb Z_2$-graded commutative ring, the $\Bbb Z_2$-graded ideals of $R$ are $(R':_{R_0}R_1)+R'$ where $R'$ is an $R_0$-submodule of $R_1$. It follows that the map $J\mapsto J\cap R_1$ realize a bijection between the set of $\Bbb Z_2$-graded ideals of $R$ and the set of $R_0$-submodule  $R_1$. 
\section{\bf $\Bbb Z_2$-Graded prime ideals}
\begin{definition}
Let $R$ be a $G$-graded commutative ring and $P$ be a $G$-graded ideal of $R$. Then $P$ is called $G$-graded prime ideal  (or homogeneous prime ideal) of $R$ if $P\ne R$ and if whenever $r$ and $s$ are homogeneous elements of $R $ such that $rs\in P$, then either $r\in P$ or $s\in P$
\end{definition}
\begin{remark}\label{primeGprime} A $G$-graded prime ideal of a $G$-graded ring $R$ is not necessarily a prime ideal. Let $R=\Bbb Z+i\Bbb Z$ (where $i^2=-1$) as a $\Bbb Z_2$-graded ring with $R_0=\Bbb Z$ and $R_1=i\Bbb Z$. It is easy to see that  $P:=2\Bbb Z+i2\Bbb Z$ is a $\Bbb Z_2$-graded prime ideal of $R$. But $P$ is not a prime ideal of $R$, since $(3+i)(3-i)=10\in P$ and $3+i,3-i\not\in P$. 
\end{remark}
Recall from \cite{Lu} that if $A$ is a commutative ring and $M$ is an $A$-module, then a prime submodule of $M$ is a proper submodule $N$ of $M$ having the property that $am\in N$ implies that $a\in (N:_SM)$ or $m\in N$ for each $a\in S$ and $m\in M$.\par 
The following Theorem describes the $\Bbb Z_2$-graded prime ideals of $R$ from prime ideals of $R_0$ and prime submodules of $R_1$.    
\begin{theorem}\label{Gprime}
Let $Q$ be a $\Bbb Z_2$-graded ideal of $R$. Then $Q$ is a $\Bbb Z_2$-graded prime ideal of $R$ if and only if 
\begin{enumerate}
\item either $Q=p+R_1$ where $p\in V(R_1^2)$, or
\item $Q=(R':_{R_0}R_1)+R' $ where $R'$ is a prime submodule of $R_1$ such that $R_1^3\not\subseteq R'$. 
\end{enumerate}
\end{theorem}
\begin{proof}
($\Leftarrow$): Let $p$ be a prime ideal of $R_0$ containing $R_1^2$ . Then $p+R_1$ is $\Bbb Z_2$-graded ideal of $R$ and  is a prime ideal of $R$ since $R/(p+R_1)=R_0/p$ is an integral domain, so it is a $\Bbb Z_2$-graded prime ideal of $R$.\par 
Let $R'$ be prime submodule of $R_1$ such that $R_1^3\not\subseteq R'$. Then $(R':_{R_0}R_1)+R'$ is a graded ideal of $R$. Let $a,b\in R_0$ such that $ab\in (R':_{R_0}R_1)+R'$, then $ab\in (R':_{R_0}R_1)$. Since $R'$ is a prime submodule of $R_1$, $(R':_{R_0}R_1)$ is a prime ideal of $R_0$. Thus  $a\in (R':_{R_0}R_1)$ or $b\in (R':_{R_0}R_1)$. If $a\in R_0 $ and $r\in R_1$ such that $ar\in (R':_{R_0}R_1)+R'$, then $ar\in R'$. It follows that $a\in (R':_{R_0}R_1)$ or $r\in R'$ since $R'$ is a prime submodule of $R_1$. Now, let $r,r'\in R_1 $ such that $rr'\in (R':_{R_0}R_1)+R'$. Assume that  $r,r'\not\in R'$. We have $rr'\in (R':_{R_0}R_1)$, so  $rr'R_1^2\subseteq (R':_{R_0}R_1)$, that is $(rR_1)(r'R_1)\subseteq (R':_{R_0}R_1)$. Since $rR_1$ and $r'R_1$ are ideals of $R_0$ and $(R':_{R_0}R_1)$ is a prime ideal of $R_0$, it follows that $rR_1\subseteq (R':_{R_0}R_1)$ or $r'R_1\subseteq (R':_{R_0}R_1) $, that is $rR_1^2\subseteq R'$ or $r'R_1^2\subseteq R'$. Since $r,r'\not\in R'$ and $R'$ is a prime submodule, we get $R_1^2\subseteq (R':_{R_0}R_1)$, so that  $R_1^3\subseteq R'$, a contradiction. Thus $r\in R'$ or $r'\in R'$. 
\par 
($\Rightarrow$): Let $Q$ be a graded prime ideal of $R$. Then, by  Proposition \ref{Gideal}, $Q=p+R'$ where $p$ is an ideal of $R_0$ and $R'$ is a submodule of $R_1$ such that $pR_1\subseteq R' $ and $R'R_1\subseteq p$. It is easy to see that $p$ is a prime ideal of $R_0$. Indeed, if $a,b\in R_0$ such that $ab\in p$, then $ab\in Q$, so $a\in Q$ or $b\in Q$. That is $a\in p$ or $b\in p$.\par 
First case,  $R_1\subseteq Q$: In this case $R_1\subseteq R'$ that is $R'=R_1$. Therefore $Q=p+R_1 $. Moreover $R_1^2\subseteq Q\cap R_0=p$. That is $p$ is a prime ideal of $R_0$ containing $R_1^2$. \par 
Second case $R_1\not\subseteq Q$: In this case $R'\ne R_1$. Let $a\in R_0$ and $r\in R_1$ such that $ar\in R'$. We see that $ar\in Q$, so $a\in Q$ or $r\in Q$ that is $a\in p\subseteq (R':_{R_0}R_1)$ or $r\in R'$. It follows that $R'$ is a prime submodule of $R_1$. Clearly, $p\subseteq (R':_{R_0}R_1)$. Let $a\in (R':_{R_0}R_1)$. Since $R'\ne R_1$, there exist $b\in R_1$ such that $b\not\in R'$. But $ab\in R'$, so that $ab\in Q $. Since $Q$ is a $\Bbb Z_2$-graded prime ideal of $R$ and $a,b$ are homogeneous elements, $a\in Q$ or $b\in Q$. Thus $a\in p=Q\cap R_0$ since $b\not\in R'$. It follows that $p=(R':_{R_0}R_1)$.
Now, if $R_1^3\subseteq R'$, then $R_1^2\subseteq (R':_{R_0}R_1)+R'=Q$. Pick $b\in R_1-R'$, then $b^2\in Q$, so $b\in Q\cap R_1=R'$, a contradiction. Thus $R_1^3\not\subseteq R'$.  
\end{proof}
In oder to determine the homogeneous spectrum of a strongly $\Bbb Z_2$-graded commutative ring, we need the following result. 
\begin{proposition}\label{Comax}
Let $R$ be a $\Bbb Z_2$-graded commutative ring. 
\begin{enumerate}
\item If $p$ is a prime ideal of $R_0$ such that $R_1^2\not\subseteq p$, then $(pR_1:_{R_0}R_1)=p$.
\item Let $Q$ be a $\Bbb Z_2$-graded ideal of $R$ such that $R_1^2+Q\cap R_0=R_0$. Then $Q$ is a graded prime ideal if and only if $Q=p+pR_1$ where $p$ is a prime ideal of $R_0$. 
\end{enumerate}
\end{proposition}
\begin{proof}
\begin{enumerate}
\item Clearly, $p\subseteq (pR_1:_{R_0}R_1)$. Let $a\in (pR_1:_{R_0}R_1)$, then $aR_1\subseteq pR_1$, so $aR_1^2\subseteq pR_1^2\subseteq p$, so that $a\in p$ since $R_1^2\not\subseteq p$. 
\item Since $R_1^2+Q\cap R_0=R_0$, $1=\displaystyle\sum_{i=1}^na_ib_i+\alpha$ where $a_i,b_i\in R_1$ and $\alpha\in Q\cap R_0$. If $Q$ is a graded prime ideal of $R$, then $R_1\not\subseteq Q$ since $R_1^2\not\subseteq Q$. By Theorem \ref{Gprime}, $Q=(R':_{R_0}R_1)+R'$ where $R'$ is a prime submodule of $R_1$ with $R_1^3\not\subseteq R_1$. We see that $p:=(R':_{R_0}R_1)=Q\cap R_0$ is a prime ideal and $pR_1\subseteq R'$. Let $r\in R'$, then $r=\displaystyle\sum_{i=1}^n(b_ir)a_i +\alpha r$. If $r'\in R_1$, then $(b_ir)r'=(b_ir')r\in R'$. It follows that $b_ir\in (R':_{R_0}R_1)=p$. So $r= \displaystyle\sum_{i=1}^n(b_ir)a_i +\alpha r\in pR_1$. Thus $R'=pR_1$, that is $Q=p+pR_1$. Conversely, assume that $Q=p+pR_1$ where $p$  is a prime ideal of $R_0$, in fact $p=Q\cap R_0$. Since $R_1^2+p=R_0$, $R_1^2\not\subseteq p$. By the previous result $p=(pR_1:_{R_0}R_1)$. Now, it is enough to show that $pR_1$ is a prime submodule of $R_1$ and $R_1^3\not\subseteq pR_1$. Let $a\in R_0$ and $r\in R_1$ such that $ar\in pR_1$. We have $r=\displaystyle\sum_{i=1}^n(b_ir)a_i+\alpha r$, so that $ar=\displaystyle\sum_{i=1}^n(ab_ir)a_i+\alpha ar $. We see that $a(b_ir)=(ar)b_i\in (pR_1)R_1\subseteq p$, so $a\in p$ or $b_ir\in p$ for all $i$. If $a\in p$, then  $a\in (pR_1,R_1)$. If $a\not\in p$, then $r=\displaystyle\sum_{i=1}^n(b_ir)a_i+\alpha r\in pR_1$. On other hand $pR_1\ne R_1$ since $(pR_1:R_1)=p\ne R_0$. Thus $pR_1$ is a prime submodule of $R_1$. If  $R_1^3\subseteq pR_1$, then $(R_1^2)^2\subseteq pR_1^2\subseteq p$, and in this case $R_1^2\subseteq p$, which is not compatible with the hypothesis. It follows that $p+pR_1$ is a $\Bbb Z_2$-graded prime ideal of $R$.     
\end{enumerate} 
\end{proof}
Now, we have the following Corollary, which describes the homogeneous spectrum of a strongly $\Bbb Z_2$-graded commutative ring.
\begin{corollary}\label{StrGprime}
Let $R$ be a strongly $\Bbb Z_2$-graded commutative ring. Then 
$$\Bbb Z_2\mathrm{Spec}R=\{p+pR_1\ / \ p\in \mathrm{Spec} R\} .$$
\end{corollary}
\begin{proof}
This follows from the previous Proposition and the fact that, for every prime ideal $p$ of $R_0$, $R_1^2+p=R_0$. 
\end{proof}
The following Corollary point out a class of $\Bbb Z_2$-graded commutative ring for which prime ideals and $\Bbb Z_2$-graded prime ideals coincide.   
\begin{corollary}
Let $R$ be a $\Bbb Z_2$-graded ring such that $R_1^2\subseteq \sqrt{0}$. Then 
\begin{enumerate}

\item The $\Bbb Z_2$-graded prime of $R$ are $p+R_1$ where $p$ is a prime ideal of $R_0$. 
\item An ideal of $R$ is $\Bbb Z_2$-graded prime if and only if it is prime.
\end{enumerate}
\end{corollary}
\begin{proof}
\begin{enumerate}
\item Immediate from \ref{Gprime}.
\item If $Q$ is a $\Bbb Z_2$-graded prime ideal of $R$, then $R_1R_1\subseteq Q$, so $R_1\subseteq Q$. By Theorem \ref{Gprime}, $Q=p+R_1$ where $p$ is a prime ideal of $R_0$. It is easy to see $R/Q=R_0/p$ is an integral domain, so $Q$ is a prime ideal of $R$. Let $P$ be any prime ideal of $R$ and set $p=P\cap R_0$. Since $R_1^2\subseteq \sqrt{0}\subseteq p\subseteq Q$, $R_1\subseteq Q$. Therefore $p+R_1\subseteq Q$. If $z=a+r\in Q$ where $a\in R_0$ and $r\in R_1$, then $a=z-r\in Q\cap R_0=p$, so $z\in p+R_1$. It follows that $Q=p+R_1$ is a $\Bbb Z_2$-graded prime ideal of $R$. 

\end{enumerate}
\end{proof}
\begin{example}
Let $R$ be a commutative ring and $M$ be an $R$-module. Consider the trivial ring extension $S=R\propto M$ as $\Bbb Z_2$-graded ring with $S_0=R$ and $S_1=M$. Clearly, $S_1^2=0$. Therefore the $\Bbb Z_2$-graded prime ideals and prime ideals coincides.    
\end{example}
An immediate consequence of the  Corollary \ref{StrGprime} is as follow;  if $R$ is a strongly $\Bbb Z_2$-graded commutative ring, then the homogeneous spectrum of $R$ and the prime spectrum of $R_0$ are homeomorphic with respect to the Zariski topologies. The next Theorem extend this result by dropping the strongly graded hypothesis.  
\begin{theorem}\label{homeo}
Let $R$ be a $\Bbb Z_2$-graded commutative ring. Then the map $\varphi: \Bbb Z_2\mathrm{Spec}R\to \mathrm{Spec}R_0$, $\varphi(P)=P\cap R_0$ is an homeomorphism of topological spaces with respect to the Zariski topologies. 
\end{theorem}
\begin{proof}
Let $P,Q\in \Bbb Z_2\mathrm{Spec}R$ with $P\cap R_0=Q\cap R_0$. If $x\in P\cap R_1$, then $x^2\in P\cap R_0=Q\cap R_0$, so $x^2\in Q$, hence $x\in Q\cap R_1$. It follows that $P=P\cap R_0+P\cap R_1\subseteq Q\cap R_0+Q\cap R_1=Q$. By the same argument we have $Q\subseteq P$, thus $P=Q$.\par 
Let $p\in \mathrm{Spec}R_0$. Consider $R'=\{x\in R_1\ / xR_1\subseteq p\}$. Then $R'$ is clearly an $R_0$-submodule of $R_1$. \\ 
First case, $R_1^2\subseteq p$. By Theorem \ref{Gprime}, $p+R_1$ is a $\Bbb Z_2$-graded prime ideal of $R$ and $\varphi(p+R_1)=p$.\\ 
Second case, $R_1^2\not\subseteq p$. Note that using the Theorem \ref{Gprime}, it enough to show that $R'$ is a prime submodule of $R_1$ with   $R_1^3\not\subseteq  R'$ and $p=(R':_{R_0}R_1)$. If  $R_1^3\subseteq R'$, then $R_1^2R_1^2=R^4\subseteq p$, since $p$ is prime, we get $R_1^2\subseteq p$, a contradiction, so $R_1^3\not\subseteq R'$. Let $a\in R_0$ and $x\in R_1$ such that $ax\in R'$. Then $axR_1\subseteq p$. Since $p$ is a prime ideal, $a\in p$ or $xR_1\subseteq p$. If $xR_1\subseteq p$, then $x\in R'$. If $a\in p$, then $(aR_1)R_1\subseteq pR_1^2\subseteq p$, so $aR_1\subseteq R'$ that is $a\in (R':_{R_0}R_1)$. Therefore $R'$ is  prime submodule of $R_1$.  Next, we show that $(R':_{R_0}R_1)=p$. We have $(pR_1)R_1=pR_1^2\subseteq p$, so $pR_1\subseteq R'$, that is $p\subseteq (R':_{R_0}R_1)$. Now, let $a\in (R':_{R_0}R_1)$. Then $aR_1\subseteq R'$, that is $aR_1^2=(aR_1)R_1\subseteq p $. Since $p$ is a prime ideal of $R_0$ and $R_1^2\not\subseteq p$, we get $a\in p$. Thus $(R':_{R_0}R_1)=p$. Finally, $p+R'$ is a $\Bbb Z_2$-graded prime ideal of $R$ and $\varphi(p+R')=p $.\par 
If $a\in R_0$, then $\varphi^{-1}(V(a))=V(a)$. Thus $\varphi$ is a continuous map.\par 
If $r\in R$ is a homogeneous elements, then $\varphi(V(r))=V(r^2)$. Thus $\varphi^{-1}$ is also a continuous map. It follows that $\varphi$ is an homeomorphism with respect to the Zariski topologies.  
\end{proof}
Next, we list a consequences of the previous Theorem.  Recall that, a $G$-graded commutative ring has Noetherian homogeneous spectrum if $G\mathrm{Spec}R$ is a Noetherian topological space with respect to the Zariski topology. The homogeneous dimension of $R$, denoted by $h\dim R$, is the maximum length $n$ of a chain $P_0\subset P_1\subset\ldots\subset P_n$ of homogeneous prime ideals of $R$.
\begin{corollary}
Let $R$ be $\Bbb Z_2$-graded commutative ring. The following statements hold.
\begin{enumerate}
\item $R$ has Noetherian homogeneous spectrum if and only if $R_0$ has Noetherian spectrum. 
\item $h\dim R=\dim R_0$.
\end{enumerate} 
\end{corollary}
\begin{proof}
\begin{enumerate}
\item Immediate.
\item In $ P_0\subset P_1\subset\ldots\subset P_n$ is a chain of homogeneous ideals of $R$, then $\varphi(P_0)\subset \varphi(P_1)\subset\ldots\subset \varphi(P_n)$ is a chain of prime ideals of $R_0$. Conversely, if $p_0\subset p_1\subset\ldots\subset p_n$ is a chain of prime ideals of $R_0$. For each $i$, $p_i=\varphi(P_i)$ for some $P_i\in \Bbb Z_2\mathrm{Spec}R$, we see that $ P_0\subset P_1\subset\ldots\subset P_n$. it follows that $h\dim R=\dim R_0$.   
\end{enumerate}
\end{proof}
We close this section by a discuss of the homogeneous radical of a graded ideal. Recall that, if $I$ is a $G$-graded ideal of a $G$-graded commutative ring $R$, then the graded radical of $I$, denoted by $\mathrm{Grad}(I)$, is 
$$\mathrm{Grad}(I)=\{x=\sum_{g\in G}x_g\ / x_g\in \sqrt{I} \text{ for all} g\in G\}$$ 
Note that $\mathrm{Grad}(I)$ is a $G$-graded ideal of $R$. 
\begin{lemma}
Let  $R$ be a $G$-graded commutative ring and $I$ be a $G$-graded ideal of $R$. Then 
$$\mathrm{Grad}(I)=\cap_PP$$
Where $P$ runs through all homogeneous prime ideals of $R$ containing $I$.
\end{lemma} 
\begin{proof}
It is easy to see that $\mathrm{Grad}(I)\subseteq \cap_PP$. Let $x\in \cap_PP$ and assume that $x\not\in \mathrm{Grad}(I)$. Then there exists $x_{g}$ ($g$-component of $x$) for some $g\in G$ such that $x_g\not\in \sqrt{I}$, that is $x_g^k\not\in I$ for all $k\in \Bbb N$. Consider the homogeneous multiplicative subset $S:=\{x_g^k\ / k\in \Bbb N\}$. We see that $I\cap S=\emptyset$. By [Lemma $4.7$,\cite{Wu}], there exists a graded prime ideal $P$ such that $I\subseteq P$ and $P\cap S=\emptyset$, that is $I\subseteq P$ and $x_g^k\not\in P$ for all $k\in \Bbb N$, a contradiction.
\end{proof}
For an ideal $I_0$ of $R_0$, denote $R_1[I_0]$ the subset $R_1[I_0]=\{x\in R_1\ / x^2\in I_0\}$.
\begin{lemma}
Let $R$ be a $\Bbb Z_2$-graded commutative ring and $P$ be a $\Bbb Z_2$-graded prime ideal of $R$. Then 
$$P=p+R_1[P]$$
Where $p=P\cap R_0$.
\end{lemma}
\begin{proof}
As in the proof of Theorem \ref{homeo}, we have $P=p+R'$ where $p=P\cap R_0$ and $R'=\{x\in R_1\ / xR_1\subseteq p\}$. So, it is enough to show that $R'=R_1[p]$. Clearly, $R'\subseteq R_1[p]$. Let $x\in R_1[p]$. Then $(xR_1)^2=x^2R_1^2\subseteq p$,  hence $xR_1\subseteq p$ since $p$ is a prime ideal of $R_0$ and $xR_1$ is an ideal of $R_0$. Thus $x\in R'$.
\end{proof}
\begin{theorem}
Let $R$ be a $\Bbb Z_2$-graded commutative ring and $J$ be a $\Bbb Z_2$-graded ideal of $R$. Then 
$$\mathrm{Grad}(J)=\sqrt{J_0}+R_1[J_0]$$
Were $J_0=J\cap R_0$.
\end{theorem}
\begin{proof}
Note that if $P$ is a $\Bbb Z_2$-graded ideal of $R$, then $J\subseteq P$ if and only if $J_0\subseteq p$ where $p=P\cap R_0$. Thus \begin{eqnarray*}
 \mathrm{Grad}(J)&=&\cap_{J_0\subseteq p}(p+R_1[p])=\cap_{J_0\subseteq p}p+\cap_{J_0\subseteq p}R_1[p]\\ &=&\sqrt{J_0}+\cap_{J_0\subseteq p}R_1[p]
 \end{eqnarray*}
On other hand, 
\begin{eqnarray*} \cap_{J_0\subseteq p}R_1[p]&=&\{x\in R_1\ / x^2\in p, \text{ for all } p\supseteq J_0\}\\ &=&\{x\in R_1\ / x^2\in \sqrt{J_0}\}=R_1[J_0]
\end{eqnarray*}
It follows that $R'=R_1[p]$ as desired.
\end{proof}
\section{\bf $\Bbb Z_2$-Maximal graded ideals}
\begin{definition}
Let $R$ be a $G$-graded ring. A graded ideal $M$ of $R$ is said to be graded maximal ideal of $R$ if $M\ne R$ and if $J$ is a graded ideal of $R$ such that $M\subseteq J\subseteq R$, then $J=M$ or $J=R$.
\end{definition}
\begin{example}
Let $R=\Bbb Z+i\Bbb Z$ be a $\Bbb Z_2$-graded ring with  $R_0=\Bbb Z$ and $R_1=i\Bbb Z$. Let $M=2\Bbb Z+i2\Bbb Z$. Then $M$ is a $\Bbb Z_2$-graded maximal ideal of $R$ and it is not a maximal (prime ) ideal of $R$. 
\end{example}
\begin{proof}
Let $J$ be a $\Bbb Z_2$-graded ideal of $R$ with $M\subsetneq J\subseteq R$. Let $x=a+ib\in J-M$ with $(a,b)\in \Bbb Z^2$. Then $a\not\in 2\Bbb Z$ or $b\not\in 2\Bbb Z$. If $a\not\in \Bbb Z$, then $a=2k+1$ where $k\in \Bbb Z$. So $1=a-(2k)\in J$, hence $J=R$. If $b\not\in 2\Bbb Z$, then $b=(2l)+1$ where $l\in \Bbb Z$, so $i= ib- 2il \in J$. Hence $-1=i^2\in J$. Thus $J=R$. Moreover $M$ is not a prime ideal see the Remark  \ref{primeGprime}, hence it is not a maximal ideal.  
\end{proof}
In the following Theorem we give a characterization of $\Bbb Z_2$-graded maximal ideals of $R$.
\begin{theorem}
Let $R$ be a $\Bbb Z_2$-graded commutative ring and $Q$ be a $\Bbb Z_2$-graded ideal of $R$. Then $Q$ is a $\Bbb Z_2$-graded maximal ideal of $R$ if and only if 
\begin{enumerate}
\item either $Q=p+R_1$ where $p\in V(R_1^2)$ is a maximal ideal of $R_0$, or 
\item $Q=(R':_{R_0}R_1)+R'$ where $R'$ is a maximal  submodule of $R_1$ such that $R_1^3\not\subseteq R'$.
\end{enumerate} 
\end{theorem}
\begin{proof}
($\Rightarrow $): Let $Q$ be a $\Bbb Z_2$-graded maximal ideal of $R$, in particular $Q$ is a $\Bbb Z_2$-graded prime ideal of $R$.\par 
First case; $R_1\subseteq Q$, from Theorem \ref{Gprime},  $Q=p+R_1$ where $p\in V(R_1^2)$. Let $p'$ be a maximal ideal containing $p$, then $Q\subseteq p'+R_1$ and $p'+R_1$ is a proper $\Bbb Z_2$-graded ideal of $R$, by maximality, $Q=p'+R_1$, so $p=p'$ is a maximal ideal of $R_0$ containing $R_1^2$. \par 
Second case; $R_1\not\subseteq Q$, by Theorem \ref{Gprime}, $Q=(R':_{R_0}R_1)+R'$ where $R' $ is a prime $R_0$-submodule of $R_1$ such that $R_1^3\not\subseteq R'$. Let $R''$ be an $R_0$-submodule of $R_1$ such that $R'\subseteq R''$ and consider the $\Bbb Z_2$-graded ideal $J:=(R'':_{R_0}R_1)+R''$. Clearly $Q\subseteq J$, by maximality $J=Q$ or $J=R$, that is $R''=R'$ or $R''=R_1$. Thus $R'$ is a maximal submodule of $R_1$.
 \\
($\Leftarrow$): Let $p$ be a maximal ideal of $R_0$ containing $R_1^2$ and $Q:=p+R_1$. Let $J$ be a $\Bbb Z_2$-graded ideal of $R$ such that $Q\subseteq J$, then $R_1=Q\cap R_1\subseteq J\cap R_1$, so $J\cap R_1=R_1$. Also $p=Q\cap R_0\subseteq J\cap R_0$. Since $p$ is a maximal ideal of $R_0$, $J\cap R_0=p$ or $J\cap R_0=R_0$ that is $J=Q$ or $J=R$. Thus $Q$ is a $\Bbb Z_2$-graded maximal ideal of $R$.\par 
Let $R'$ be a maximal submodule of $R_1$ such that $R_1^3\not\subseteq R'$ and $Q:=(R':_{R_0}R_1)+R'$. Clearly $Q$ is a proper $\Bbb Z_2$-graded ideal of $R$. Now, let $P$ be any $\Bbb Z_2$-graded maximal ideal of $R$ containing $Q$. Denote $p':=P\cap R_0$. We have $(R':_{R_0}R_1)=Q\cap R_0\subseteq P\cap R_0$. Since $R'$ is a maximal submodule of $R_1$, $(R':_{R_0}R_1)$ is a maximal ideal of $R_0$, it follows that $ p'=(R':_{R_0}R_1)$ or $p'=R_0$. If $p'=R_0$, then $P=R$ since $1\in P$, with is not compatible with the fact that $P$ is a $\Bbb Z_2$-graded maximal ideal. So  $p'=(R':_{R_0}R_1)$, hence $\varphi(Q)=\varphi(P)$ where $\varphi$ is the map defined in Theorem \ref{homeo}, by the said Theorem, we have $P=Q$. Thus $Q$ is a $\Bbb Z_2$-graded maximal ideal of $R$.    
\end{proof}
Next, we establish a relation between $\Bbb Z_2$-graded maximal ideal of $R$ and maximal ideal of $R_0$. Before, we establish the following Lemma. 
\begin{lemma}
Let $R$ be a $\Bbb Z_2$-graded ring and $R'$ be an $R_0$-submodule of $R_1$ such that $R_1^3\not\subseteq R'$. The followings statements are equivalent.
\begin{enumerate}
\item $R'$ is a maximal submodule of $R_1$.
\item $(R':_{R_0}R_1)$ is maximal ideal of $R_0$.
\end{enumerate} 
In this case $R'=(R':_{R_0}R_1)R_1$.
\end{lemma}
\begin{proof}
If $R'$ is a maximal submodule of $R_1$, then clearly $(R':_{R_0}R_1)\ne R_0$. Pike $\alpha\in R_0-(R':_{R_0}R_1)$, then $\alpha r\not\in R'$ for some $r\in R_1$. It follows that $R'\subsetneq R'+R_0(\alpha r)$, so $ R'+R_0(\alpha r)=R_1$. In particular, $r=\beta \alpha r+r'$ where $\beta\in R_0$ and $r(\in R'$. So that $(1-\beta \alpha)r=r'\in R'$. Thus $1-\beta \alpha\in (R':_{R_0}R_1)$, that is $(R':_{R_0}R_1)+(\alpha)=R_0$. Hence $(R':_{R_0}R_1)$ is  a maximal ideal of $R_0$. For the converse, assume that  $(R':_{R_0}R_1)$ is a maximal ideal of $R_0$. Let $R''$ be a submodule of $R_1$ with  $R'\subseteq R''$. Then $(R':_{R_0}R_1)\subseteq (R'':_{R_0}R_1)$, hence $(R'':_{R_0}R_1)=R_0$ or $(R'':_{R_0}R_1)=(R':_{R_0}R_1)$. If $(R'':_{R_0}R_1)=R_0 $, then clearly $R''=R_1$. Assume that $(R'':_{R_0}R_1)=(R':_{R_0}R_1)$. Let $x\in R''$, then $xR_1\subseteq (R'':_{R_0}R_1)=(R':_{R_0}R_1)$, so $xR_1^2\subseteq R'$. Since $R'$ is a prime submodule of $R_1$ and $R_1^2R_1=R_1^3\not\subseteq R'$, it follows that $x\in R'$. Thus $R''=R'$. Therefore $R'$ is maximal submodule of $R_1$.\par 
We see that $R_1^2\not\subseteq (R':_{R_0}R_1)$, so that $(R':_{R_0}R_1)+R_1^2=R_0$ since $(R':_{R_0}R_1)$ is a maximal ideal of $R_0$. Thus $1=\displaystyle\sum_{k=1}^na_ib_i+\alpha$ where $a_k,b_k\in R_1$ and $\alpha\in (R':_{R_0}R_1)$. Now, if $r\in R'$, then $r=\displaystyle\sum_{k=1}^nra_ib_i+\alpha r\in (R':_{R_0}R_1)R_1$ since $a_ir\in (R':_{R_0}R_1)$ and $\alpha \in (R':_{R_0}R_1)$. Therefore $R'=(R':_{R_0}R_1)R_1$.   
\end{proof}
Now, we are able to establish a relation between $\Bbb Z_2$-graded maximal ideals of $R$ and maximal ideals of $R_0$, as follow.
\begin{theorem}\label{Gmax}
Let $R$ be a $\Bbb Z_2$-graded commutative ring. Then $$\Bbb Z_2\mathrm{Max} R=\{ p+R_1/\ p\in\mathrm{Max}R_0\cap V(R_1^2) \}\cup \{ p+pR_1 \ /\ p\in \mathrm{Max }R_0\cap D(R_1^2)\}$$
\end{theorem} 
\begin{proof}
Let $R'$ be a maximal submodule of $R_1$ such that $R_1^3\not\subseteq R'$ and denote $p:=(R':_{R_0}R_1)$. By the previous Lemma, $p$ is a maximal ideal of $R_0$ and $R'=pR_1$. Moreover $R_1^2\not\subseteq p$ since $R_1^3\not\subseteq R'$. Thus $(R':_{R_0}R_1)+R'=p+pR_1$. Conversely, let $p$ be a maximal ideal of $R_0$ such that $R_1^2\not\subseteq p$ and denote $R'=pR_1$. By the Proposition \ref{Comax}, $(R':_{R_0}R_1)=p$ is a maximal ideal of $R_0$.  Since $R_1^2\not\subseteq p $, $R_1^3\not\subseteq R'$. It follows from the previous Lemma that $R'$ is a maximal submodule of $R_1$.
\end{proof}
An immediate consequence, we have the following Corollary.
\begin{corollary}
Let $R$ be a strongly $\Bbb Z_2$-graded commutative ring. Then 
$$\Bbb Z_2\mathrm{Max} R=\{p+pR_1\ / \ p\in \mathrm{Max}R_0 \}$$
\end{corollary}
\begin{proof}
Follows immediately from the previous Proposition. 
\end{proof}
\begin{corollary}
Let $R$ be a $\Bbb Z_2$-graded commutative ring. Then 
\begin{enumerate}
\item $\varphi$ induce an homeomorphism between $\Bbb Z_2\mathrm{Max}R$ and $\mathrm{Max}R_0$. 
\item $R$ is a $\Bbb Z_2$-graded local ring if and only if $R_0$ is a local ring.
\end{enumerate}
\end{corollary}
\begin{proof}
\begin{enumerate}
\item Due to the fact that, $\Bbb Z_2$-graded maximal ideals of $R$ (respectively, maximal ideals of $R_0$) are exactly the closed points of $\Bbb Z_2\mathrm{Spec} R$ (respectively $\mathrm{Spec}R_0$) with respect to the Zariski topology.
\item This is the case of one closed point.
\end{enumerate}
\end{proof}
\section{\bf $\Bbb Z_2$-Graded field}
\begin{definition}
Let $R$ be a $G$-graded ring. 
\begin{enumerate}
\item $R$ is called a $G$-graded integral domain if whenever $r$ and $s$ are homogeneous elements of $R $ such that $rs=0$, then either $r=0$ or $s=0$. 
\item $R$ is called a $G$-graded field if every nonzero homogeneous element is a unit.
\end{enumerate}
\end{definition}
\begin{remark} 
A $G$-graded commutative ring $R$ is  a $G$-graded domain if and only if the zero ideal is a $G$-graded prime ideal of $R$. Whereas,  a $G$-graded commutative ring is a $G$-graded field if and only if the zero ideal is a $G$-graded maximal ideal of $R$. 
\end{remark}
\begin{theorem}
Let $R$ be a $\Bbb Z_2$-graded commutative ring. Then $R$ is a $\Bbb Z_2$-graded integral domain if and only if $R_1=0$ and $R_0$ is an  integral domain or $0 $ is  a prime submodule of $R_1$ such that  $(0:_{R_0}R_1)=0$ and $R_1^3\ne 0$. 
\end{theorem}
\begin{proof}
$(\Rightarrow)$ If $R$ is a $\Bbb Z_2$-graded integral domain, then $0$ is a $\Bbb Z_2$-graded prime ideal of $R$. By Theorem \ref{Gprime}, $0=p+R_1$ where $p\in V(R_1^2)$ or $0=(R':_{R_0}R_1)+R'$ where $R'$ is a prime submodule of $R_1$ with $R_1^3\not\subseteq R'$. If $0=p+R_1$, then $R_0$ and $p=0$ is a prime ideal of $R_0$, that is $R_1=0$ and $R_0$ is an integral domain. If $ 0=(R':_{R_0}R_1)+R'$, then $ R'=0$ and $(R':_{R_0}R_1)=0$, as desired. \\
$(\Leftarrow)$ Immediate. 
\end{proof}
\begin{corollary}
Let $R$ be a strongly $\Bbb Z_2$-graded commutative ring. Then $R$ is a $\Bbb Z_2$-graded integral domain if and only if $R_0$ is an integral domain.
\end{corollary}
\begin{proof}
By the Corollary \ref{StrGprime}.
\end{proof}
\begin{theorem}
Let $R$ be a $\Bbb Z_2$-graded commutative ring. Then $R$ is a $\Bbb Z_2$-graded field if and only if $R_0$ is a field and either $R_1=0$ or  $R_1^2\ne 0$.
\end{theorem}
\begin{proof}
$R$  is a $\Bbb Z_2$-graded field if and only if $0$ is a $\Bbb Z_2$-graded maximal ideal of $R$. According to Theorem \ref{Gmax}, $0=p+R_1$ with $p$ maximal ideal of $R_0$ containing $R_1^2$, which is equivalent to $R_1=0$ and $0=p$ is a maximal ideal of $R_0$. Or $0=p+pR_1$ with $p$ is a maximal ideal of $R_0$ such that $R_1^2\not\subseteq p$, which is equivalent to $0=p$ is a maximal ideal of $R_0$ and $R_1^2\ne 0$. 
\end{proof}
\begin{corollary}
Let $R$ be a $\Bbb Z_2$-graded commutative ring. Then $R$  is a $\Bbb Z_2$-graded field if and only if $R_0$ is a field and either $R_1=0$ or $R_1=R_0b$ where $b\in R_1$ such that $b^2\ne 0$.
\end{corollary}
\begin{proof}
We may assume that $R_1\ne 0$. If $R$ is a $\Bbb Z_2$-graded field, then  by the previous Theorem $R_0$ is a field and $R_1^2\ne 0$. So $bb'\ne 0$ for some  $b,b'\in R_1$ since $R_1^2$ is generated by all elements $bb'$ where $b,b'\in R_1$. As $bb'\ne 0$ and $R_0$ is a filed, $bb'$ is a unit in $R_0$. There exists $\alpha\in R_0$ such that $bb'\alpha=1$. It follows that $r=\alpha (rb')b$ for all $r\in R_1$. Therefore $R_1=R_0b$. Clearly $b^2\ne 0$ since $R_1^2\ne 0$. The converse is immediate, it is easy to see that every homogeneous elements in a unit.  
\end{proof}
For a commutative ring $S$ we look at $S[X]$ as a $\Bbb Z_2$-graded commutative ring see Example \ref{NGradedPolynome}. We see that $(S[X])_0=\{P(X^2)\ / \ P\in S[X]\}$ and $(S[X])_1=\{XP(X^2)\ / \ P\in S[X]\}$. For $\alpha\in S$, $(X^2-\alpha)$ is a $\Bbb Z_2$-graded ideal of $S[X]$. Therefor $\frac{S[X]}{(X^2-\alpha)}$ is a $\Bbb Z_2$-graded commutative ring and clearly $\left(\frac{S[X]}{(X^2-\alpha)}\right)_0=S$ and $\left(\frac{S[X]}{(X^2-\alpha)}\right)_1=S\overline{X}$. Note that if $S$ is a field and $\alpha\ne 0$, then $\frac{S[X]}{(X^2-\alpha)}$ is a graded field.
\begin{corollary}
Let $R$ be a $\Bbb Z_2$-graded commutative ring. If $R$ is a $\Bbb Z_2$-graded field, then $R_0$ is a field and either $R_1=0$ or  $R$  is isomorphic to  $\frac{R_0[X]}{(X^2-\alpha)}$ as $\Bbb Z_2$-graded rings, where $\alpha$ is a nonzero element of $R_0$. 
\end{corollary}
\begin{proof}
We may assume that $R_1\ne 0$. By the previous Corollary $R_0$ is a field and $R_1=R_0b$ where $b\in R_1$ such that $b^2=0$. Let $f:R_9[X]\to R$ the map defined by $ f(P)=P(b)$. Clearly $f$ is a rings morphism. We have $f(P(X^2))=P(b^2)\in R_0$ since $b^2\in R_0$, also $f(XP(X^2))=bP(b^2)\in R_1$. It follows that $f$ is a $\Bbb Z_2$-graded morphism. For $r=r_0+r_1\in R$. We see that $r_1=a b$ for some $a\in R_0$, so that $f(r_0+aX)=r_0+ab=r$. Now, let $\alpha=b^2\in R_0$, then $(X^2-\alpha)\subseteq \ker f$. Let $P\in \ker f$, by the euclidean division, $P=(X^2-\alpha)Q+dX+c$ where $Q\in R_0[X]$ and $d,c\in R_0$. But $P(b)=0$ implies that $db+c=0$, so that $d=0$ and $c=0$ ( if $d\ne 0$, then $d$ is a unit, so $b=d^{-1}c\in R_0\cap R_1$  which is not possible). It follows that $\ker f=(X^2-\alpha)$.  
\end{proof} 
\section{\bf $\Bbb Z_2$-Graded integral domain and integral domain}
For $x=a+r\in R$ where $(a,r)\in R_0\times R_1$, define $N(x):=a^2-r^2$. Then, $N(xy)=N(x)N(y)$ and $N(x)\in R_0$ for all $x,y\in R$. Also $N(0)=0$ and $N(1)=1$. Next we establish the relation between some properties of $N$ and the integrity of $R$. To do, let $\mathcal{N}:=\{x\in R\ / N(x)=0\}$.
\begin{theorem}
Let $R$ be a $\Bbb Z_2$-graded ring. The following statements are equivalent.
\begin{enumerate}
\item $R$ is an integral domain.
\item $R$ is a $\Bbb Z_2$-graded integral domain and $\mathcal{N}=0$.
\end{enumerate}
\end{theorem} 
\begin{proof}
$(1)\Rightarrow (2)$: Clearly, if $R$ is an integral domain then $R$ is a $\Bbb Z_2$-graded integral domain. Let $x=a+r\in \mathcal{N}$, then $a^2-r^2=0$, so that $(a+r)(a-r)=0$, thus $a+r=0$ or $a-r=0$. Hence $x=0$. It follows that $\mathcal{N}=0$.\\
$(2)\Rightarrow (1)$: Let $x,y\in R$ with $xy=0$. Then $N(x)N(y)=0$. Note that $N(x)$ and $N(y)$ are homogeneous, so that $N(x)=0$ or $N(y)=0$ that is $x\in \mathcal{N}$ or $y\in \mathcal{N}$. Thus $x=0$ or $y=0$. 
\end{proof}


\end{document}